\documentclass[a4paper,10pt]{article}
\usepackage[latin1]{inputenc}
\usepackage[T1]{fontenc} 
\usepackage{amsfonts}
\usepackage{amsmath,amsthm}
\usepackage{graphicx}
\usepackage{amssymb}
\usepackage[scale=0.7]{geometry}
\usepackage{tikz}
\usepackage{tikz-cd}
\usepackage{color}
\usepackage{float}

\newtheorem{theorem}{Theorem}[section]
\newtheorem{algorithm}[theorem]{Algorithm}

\newtheorem{lemma}[theorem]{Lemma}

\newtheorem{proposition}[theorem]{Proposition}

\newtheorem{remark}[theorem]{Remark}

\newcommand{\GG}{\mathbb{G}}
\newcommand{\RR}{\mathbb{R}}

\newcommand{\Mc}{\mathcal{M}}

\DeclareMathOperator{\rank}{\mathrm{rank}}

\begin{document}

\title{{A new splitting algorithm for dynamical low-rank approximation motivated by the fibre bundle structure of matrix manifolds}}
\author{M. Billaud-Friess, A. Falc\'o \& A. Nouy}
\date{\today}
\maketitle

\begin{abstract}
{In this paper, we propose a new splitting algorithm for dynamical low-rank approximation motivated by the fibre bundle structure of the set of fixed rank matrices.  
We first introduce a geometric description of the set of fixed rank matrices which relies on a natural parametrization of matrices.} More precisely, it is endowed with the structure of analytic principal bundle, with an explicit description of local charts. For matrix differential equations, we introduce a first order numerical integrator working in local coordinates. The resulting {algorithm} can be interpreted as a particular splitting of the projection operator onto the tangent space of the low-rank matrix manifold. {It is proven to be exact in some particular case. Numerical experiments confirm this result and illustrate the robustness of the proposed algorithm.}
\end{abstract}

{\small \noindent\textbf{Keywords:} Dynamical low-rank approximation, matrix manifold, matrix differential equation, splitting integrator }\\

{\small \noindent\textbf{2010 AMS Subject Classifications:} 15A23, 65F30, 65L05, 65L20}\\


\section{Introduction}
High-dimensional dynamical systems arise in variety of applications as quantum chemistry, physics, finance and uncertainty quantification, to name a few. Discretization of such problems with traditional numerical methods often leads to complex numerical problems usually untractable, in particular if they depend on parameters.  Model Order Reduction (MOR) methods aim at reducing the complexity of such problems by projecting the solution onto low-dimensional manifolds. In this paper, we particularly focus on dynamical low-rank methods. Such methods have been considered for low-rank approximation of time-dependent matrices  \cite{Koch2007,Ceruti2020,Falco2018} with possible symmetry properties \cite{Ceruti2019}, and tensors \cite{Khoromskij2012,Lubich2014,Kieri2016,Kieri2019} {or more recently extended to parabolic problems \cite{Bachmayr2020}. In the context of parameter-dependent partial differential equations, let us mention also dynamical orthogonal approximation, in a Riemannian framework \cite{Sapsis2009,Cheng2013,Musharbash2015,Feppon2018a,Feppon2018b,Feppon2019},} and also dynamical reduced basis method \cite{Billaud2017}.\\

Here, we focus on low-rank approximation of  time-dependent matrices $A(t) \in \RR^{n\times m}$. Introducing $\dot A(t) = \frac{d}{dt}A(t)$ the time derivative, the matrix $A(t)$ is defined as the solution of the following Ordinary Differential Equation (ODE) 
\begin{equation}
\dot A(t) = F(A(t),t), \qquad A(0) = A^0,
\label{eq:dynsys}
\end{equation}
given $A^0 \in \RR^{n\times m}$ and $F : \RR^{n\times m} \times [0,T]\to \RR^{n\times m}$. 
Dynamical low-rank methods aim at approximating at each instant $t$ the matrix $A(t)$  by the matrix $Z(t)$ which belongs to the nonlinear manifold of fixed rank matrices
$$
\Mc_r({\RR^{n\times m}})= \{Z \in \RR^{n\times m} : \rank(Z)=r\},
$$
where $r \ll \min(n,m)$ stands for  the rank.
When $A(t)$ is known, $Z(t)$ can be defined as the best rank-$r$ approximation solution of
\begin{equation}
 Z(t) = \arg \min_{W \in \Mc_r({\RR^{n\times m}})} \| A(t) -W\|,
\label{eq:best-approx}
\end{equation}
with $\|\cdot\|$ the Frobenius norm. In that case, $Z$  is obtained through a Singular Value Decomposition (SVD) of $A(t)$  for each instant $t$. Nevertheless, as $A$ is implicitly given by the dynamical system \eqref{eq:dynsys}, it is more relevant to introduce low-rank approximation using $\dot A$. To that goal, the approximation $Z$  is classically obtained through its derivative $\dot Z$ which satisfies the Dirac-Frenkel variational principle
\begin{equation}
\dot Z(t) = \arg \min_{\delta W \in T_{Z(t)}\Mc_r({\RR^{n\times m}})} \| \delta W  -F(Z(t),t)\|,
\label{eq:DF}
\end{equation}
given $Z(0) = Z^0 \in \Mc_r(\RR^{n\times m})$ the best rank-$r$ approximation of $A(0)$ and $T_{Z(t)}\Mc_r({\RR^{n\times m}})$ the tangent space to $\Mc_r(\RR^{n\times m})$ at $Z(t)$. Equivalently,  $\dot Z(t)$ corresponds to the orthogonal projection of $F(Z(t),t)$ (see Figure \ref{fig:proj}) on the solution dependent tangent space, i.e.
\begin{equation}
\dot Z(t)  = P_{T_Z(t)} F(Z(t),t), \qquad Z(0) = Z^0,\\
\label{eq:proj}
\end{equation}
where $P_{ T_{Z}}$ denotes the projection onto $T_{Z(t)}\Mc_r({\RR^{n\times m}})$. \\
\begin{figure}[H]
\centering
\includegraphics[scale =0.6]{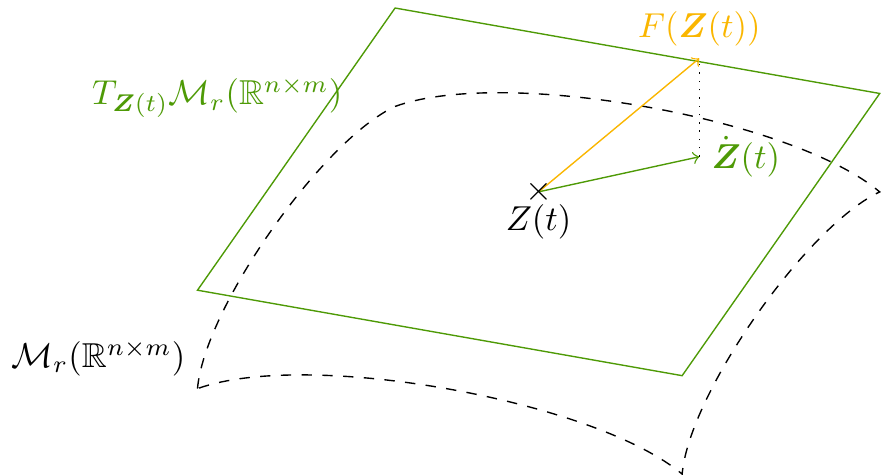}
\caption{Orthogonal projection on the tangent space $T_{Z(t)}\Mc_r({\RR^{n\times m}})$. \label{fig:proj}}
\end{figure}

In view of MOR, the goal of low-rank methods is to approximate the solution $A(t)$ of Equation \eqref{eq:dynsys} with $Z(t)$ solution of Equation \eqref{eq:proj} which is cheaper to compute. However, in practice additional difficulties appear for the numerical integration of Equation \eqref{eq:proj}. \\
The first difficulty relies on the proper description of the manifold of fixed rank matrices $\Mc_r(\RR^{n\times m})$. 
In practice, a way to compute the rank-$r$ matrix $Z(t)$ is done through its parametrization
\begin{equation}
Z(t) = U(t) G(t) V(t)^T,
\label{eq:param}
\end{equation} 
with
$U(t) \in \RR^{n\times r}, V(t) \in \RR^{m\times r}$ and $G(t) \in \RR^{r\times r}$.
Such a parametrization of the matrix $Z(t)$ is not unique. A way to dodge this undesirable property is to properly define the tangent space $T_{Z(t)}\Mc_r({\RR^{n\times m}})$. Given $Z(0)= U(0)G(0)V(0)^T$, the matrix $Z(t)$ admits a unique decomposition of the form \eqref{eq:param} when imposing the so-called {\it gauge conditions} on  $U,V$ (see \cite[Proposition 2.1]{Koch2007}). In addition, the system \eqref{eq:proj} results in a system ODEs driving the evolution of the parameters $U,G,V$.\\
The second difficulty appears when numerical integration is performed for solving the resulting system of ODEs governing the evolution of  parameters $U,V$ and $G$. 
Indeed in presence of small singular values for $Z(t)$, the matrix $G(t)$ may be ill-conditioned. As consequence, classical integration schemes may be unstable  (see e.g. \cite[Section 2.1]{Kieri2016}). Moreover, in case of {\it overapproximation}, i.e. when the approximation $Z(t)$ has a rank $r$ greater than the rank of the exact solution $A(t)$, these method fails since $G(t)$ becomes singular. A possible way to address this issue is regularize $G$ such that it remains invertible. {Nevertheless, it has the drawback to modify the problem (then the approximate solution) and does not prevent ill-conditioning of $G(t)$. }
In \cite{Lubich2014}  an explicit projection-splitting integrator is proposed to deal with numerical integration of \eqref{eq:proj}.  It is based on a Lie-Trotter splitting of the projection operator $P_{T_{Z(t)}}$. 
In addition to its simplicity, it has the advantage to remains robust in case of small singular values and especially for overapproximation as it avoids the inversion of $G(t)$.
The resulting method is fully explicit and first order, extension to second order via Strang splitting scheme could be considered. 
{Variants of this algorithm have been proposed in \cite{Ceruti2019} to integrate some symmetry properties. More recently in  \cite{Ceruti2020}\footnote{We have been aware of this reference while revising the present paper.}, the authors derive a different robust algorithm for dynamical low-rank approximation.} When interested in higher order approximation, projection based methods \cite{Kieri2019}, in the lines of Riemaniann optimization, combined to explicit Runge-Kutta schemes could be considered. Such methods work as follows. 
Perform one step of the numerical scheme leaving the manifold, and then project on the manifold by means of retraction. The latter step is usually performed using a $r$-terms truncated SVD.\\ 

In \cite{Billaud2017}, the authors give a different geometric description of the fixed-rank matrix manifold and the associated tangent space. This description combines a natural definition of the neighborhood of $Z$ together with explicit description of local charts such that the set $\Mc_r(\RR^{n\times m})$ is endowed with the structure of analytic principal bundle \cite[Theorem 4.1]{Billaud2017}.  
Moreover, it ensures that any matrix in the neighborhood of a matrix $Z$ (including itself) admits a unique representation in the form $UGV^T$. 
The main contribution of this paper is twofold. First, we revisit dynamical low-rank approximation by using the geometric description of the matrix manifold given in \cite{Billaud2017}. The resulting system of ODEs on the parameters is shown to be related to the one obtained in \cite{Koch2007} but with no need of gauge conditions. Secondly, {relying on this geometric description} of $\Mc_r(\RR^{n\times m})$, we derive a first order numerical integrator in local coordinates for solving \eqref{eq:proj} that can be interpreted as a splitting integrator. It is proven to coincide with   the so-called {\it KSL splitting algorithm} introduced in \cite{Lubich2014}  in the particular case where the flux $F$ only depends on $t$. \\

The outline of the paper is as follows. We detail in Section \ref{sec:geom-description}  the proposed geometric description of the manifold of rank-$r$ matrices.
In Section \ref{sec:algos} we describe a new splitting algorithm relying on the proposed geometric description 
Finally, in Section \ref{sec:num_results}, we confront the proposed splitting algorithm to KSL splitting integrator  \cite{Lubich2014} on several numerical test cases.

\section{Dynamical low-rank approximation: a geometric approach} \label{sec:geom-description}

Dynamical low-rank approximation consists in approximating at each time $t$ the matrix $A(t)$ by a matrix $Z(t) \in \Mc_r(\RR^{n\times m})$  that can be represented (in non-unique way) by means of the factorization 
$$
Z = U G V^T,
$$
with $U \in \Mc_r(\RR^{n\times r}), V \in \Mc_r(\RR^{m\times r})$ and $G \in \mathrm{GL}_r$ {with $ \mathrm{GL}_r$ the Lie group of $r\times r $ invertible matrices}.  We present in Section \ref{sec:geom-description1} a geometric description of $ \Mc_r(\RR^{n\times m})$. We restrict the presentation to essential  elements of geometry required thereafter. The interested reader could consult the original paper \cite{Billaud2017} for further details. In Section \ref{subsec:dlr} we discuss the interest of the proposed description for dynamical low-rank approximation. Finally, we draw  the link with the geometric description proposed in \cite{Koch2007} in Section \ref{sec:classical}.

\subsection{Chart based geometric description of $\Mc_r(\RR^{n\times m})$}\label{sec:geom-description1}

Let $Z = UGV^T \in \Mc_r(\RR^{n\times m})$. We consider $U_\bot \in \Mc_{n-r}(\mathbb{R}^{n\times (n-r)}), V_\bot \in \Mc_{m-r}(\mathbb{R}^{n\times (m-r)}),$ the matrices such that 
$U_\perp^T U = 0$ and $V_\perp^T V = 0$. The {\it neighborhood} ${\cal U}_{Z}$ of $Z$ in $\Mc_r(\RR^{n\times m})$  is defined as the set 
$$
{\cal U}_{Z} = \{(U+U_\perp X)H(V+V_\bot Y)^T :  (X,Y,H) \in \RR^{(n-r)\times r}\times \RR^{(m-r)\times r} \times  \mathrm{GL}_r\}.
$$
\begin{figure}[H]
\centering
\includegraphics[scale=1]{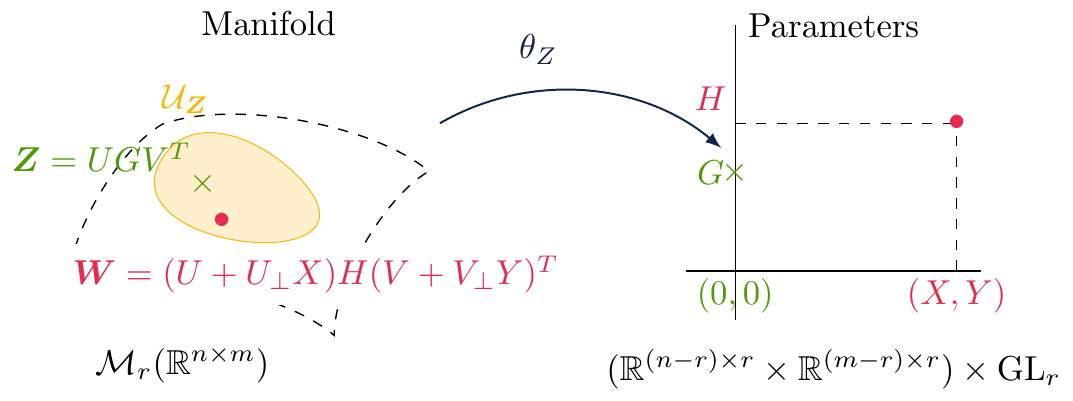}
\caption{Representation of the local chart $\theta_Z$ that associates to $W = (U + U_\perp X)H(V + V_\perp Y )^T$ in  ${\cal U}_{Z} \subset\Mc_r(\RR^{n\times m})$ the parameters $(X,Y,H) \in \RR^{(n-r)\times r}\times \RR^{(m-r)\times r} \times  \mathrm{GL}_r$ \label{fig:chartneigh}.}
\end{figure}

We associate to the neighborhood ${\cal U}_Z$ of $Z$ the {\it local chart} $\theta_Z : {\cal U}_Z \to \RR^{(n-r)\times r}\times \RR^{(m-r)\times r} \times \mathrm{GL}_r$ (see Figure \ref{fig:chartneigh})  which is given by
$$
\theta_Z(W) = (U_\bot^+W(V^+)^T(U^+W(V^+)^T)^{-1},V_\bot^+W^T(U^+)^T(V^+W^T(U^+)^T)^{-1},U^+W(V^+)^T)
$$
for any $W\in \mathcal{U}_Z$. Here $U^+$ and $V^+ $ stand for the Moore-Penrose pseudo-inverses \footnote{ For any $A \in \RR^{n \times m}$, the Moore-Penrose pseudo inverse is given by $A = (A^TA)^{-1} A^T.$} of $U$ and $V$ respectively. This means that any matrix $W$ belonging to the neighborhood ${\cal U}_Z$ admits a unique parametrization 
$$W =  \theta_Z^{-1}(X,Y,H),$$ 
with parameters $(X,Y,H) \in \mathbb{R}^{(n-r)\times r} \times \mathbb{R}^{(m-r)\times r} \times \mathrm{GL}_r$ and where the  map $\theta_Z^{-1}$ is defined by
$$\theta_Z^{-1}(X,Y,H) = (U+U_{\bot}X)H(V+V_{\bot}Y)^T.$$
In this description, the parameters are not longer $U,V,G$ but $X,Y,H$. Note that $\theta_Z^{-1}(0,0,G) = Z$.\\
  
Such geometric description confers the set $\Mc_r(\RR^{n \times m})$  the structure of an analytic manifold and of a principal bundle \cite[\S 4]{Billaud2017}.

 \begin{proposition} \label{prop:manifoldprop}
The set of fixed rank matrices $\Mc_r(\RR^{n\times m})$  equipped  with the atlas $ {\cal A}_{n,m,r} = \{({\cal U}_Z, \theta_Z) : Z \in \Mc_r(\RR^{n\times m}) \} $ is an analytic $r(n+m-r)$-dimensional manifold modelled on $\RR^{(n-r)\times r}\times \RR^{(m-r)\times r} \times \RR^{r\times r}$. Moreover $\Mc_r(\RR^{n\times m})$  is an analytic principal bundle with typical fiber $ \mathrm{GL}_r$ and base\footnote{ Here $\GG_r(\RR^p) = \{V_r \subset \RR^p : \dim(V_r)=r\}$ denotes the Grassmann manifold.} $\GG_r(\RR^n)\times \GG_r(\RR^m)$. 
\end{proposition}


We now give a description of the tangent space to the manifold of rank-$r$ matrices at $Z$ denoted $T_{Z} \Mc_r(\RR^{n\times m})$.\\

To that goal, we define 
the tangent map at  $Z \in \mathcal{M}_r(\mathbb{R}^{n\times m})$ noted $\mathrm{T}_{Z}{i}$  by
\begin{align*}
&&\mathrm{T}_{Z}{i}:  \mathbb{R}^{(n-r)\times r} \times \mathbb{R}^{(m-r)\times r} \times 
\mathbb{R}^{r\times r} &\rightarrow  \mathbb{R}^{n \times m},\\ 
&&(\delta X,\delta  Y,\delta H) &\mapsto  
U_{\bot}\delta  XGV^T  + U G (V_{\bot}\delta  Y)^T + U\delta  H V^T.
\end{align*}
Then, the  tangent space  to $\mathcal{M}_r(\mathbb{R}^{n\times m})$ at $Z$ is defined as the image through $\mathrm{T}_{Z}{i}$ of the tangent space in the local coordinates \footnote{ This means the tangent space to the local parameter space $\RR^{(n-r)\times r}\times \RR^{(m-r)\times r} \times \RR^{r\times r}$ at  $(0,0,G)$.} in $\mathbb{R}^{(n-r)\times r} \times \mathbb{R}^{(m-r)\times r}\times \mathbb{R}^{r\times r}$ 
\begin{align*}
{T}_Z \mathcal{M}_r(\mathbb{R}^{n\times m}) 
&= \{ U_{\bot} \delta X GV^T  +U G (V_\bot \delta Y)^T +  U \delta G V^T : \delta X \in \mathbb{R}^{(n-r)\times r}, \delta Y \in \mathbb{R}^{(m-r)\times r} ,\delta G \in  \mathbb{R}^{r\times r}  \}.
\end{align*}

As stated in \cite[Proposition 4.3]{Billaud2017}, $\mathrm{T}_{Z}i$  is an isomorphism between $T_Z \mathcal{M}_r(\mathbb{R}^{n\times m})$ and $\mathbb{R}^{(n-r)\times r} \times \mathbb{R}^{(m-r)\times r}\times \mathbb{R}^{r\times r}$. 

\begin{proposition}  \label{prop:fixedrank-submanifold}
The tangent map $\mathrm{T}_{Z}i$ at $Z$ is a linear isomorphism with inverse $ (\mathrm{T}_{Z}{i})^{-1}$ given by 
  $$
 (\mathrm{T}_{Z}{i})^{-1}(\delta W) = (U_\bot^+ \delta W (V^+)^TG^{-1},V_\bot^+\delta W^T (U^+)^TG^{-T},U^+ \delta W(V^+)^T)
 $$ 
 for $\delta W \in \mathbb{R}^{n \times m} $. 
\end{proposition}

By Proposition \ref{prop:fixedrank-submanifold}, any tangent matrix $\delta Z \in T_{Z} \Mc_r(\RR^{n\times m})$ admits a unique parametrization of the form 
\begin{equation}
\delta Z =  \mathrm{T}_{Z}{i} (\delta X, \delta Y, \delta H) = U_\perp\delta XGV^T+UG(V_\perp \delta Y)^T+U \delta H V^T,
\label{eq:tanchart}
\end{equation}
where $(\delta X, \delta Y, \delta H)  \in \mathbb{R}^{(n-r)\times r} \times \mathbb{R}^{(m-r)\times r}\times \mathbb{R}^{r\times r}$ are uniquely given through  
\begin{equation}
\begin{array}{rcl}
\delta X&=& U_\perp^+ \delta Z (V^+)^T G^{-1}, \\ 
\delta Y&=& V_\perp^+ \delta Z^T(U^+)^T G^{-T},\\ 
\delta H&=& U^+ \delta Z (V^+)^T.
\end{array}
\label{eq:factchart}
\end{equation}

\begin{remark} \label{rem:HV}
The tangent space could be decomposed into distinct pieces which are the vertical tangent space 
$$
{T}_Z^V \mathcal{M}_r(\mathbb{R}^{n\times m}) =  \{  U \delta G V^T : \delta G \in  \mathbb{R}^{r\times r}  \},
$$
and the horizontal tangent space
$$
{T}_Z^H \mathcal{M}_r(\mathbb{R}^{n\times m}) =  \{U_{\bot} \delta X GV^T  +U G (V_\bot \delta Y)^T   :\delta X \in \mathbb{R}^{(n-r)\times r},  \delta Y \in \mathbb{R}^{(m-r)\times r}   \},
$$
where ${T}_Z^V\mathcal{M}_r(\mathbb{R}^{n\times m})$ is associated to the fiber and ${T}_Z^H \mathcal{M}_r(\mathbb{R}^{n\times m})$ to the base.
\end{remark}

\subsection{Dynamical low-rank approximation} \label{subsec:dlr}

In the context of dynamical low-rank approximation, we recall that $Z$ is given through the projected Equation \eqref{eq:proj}. By definition of the tangent space \eqref{eq:tanchart},  the tangent matrix is given by $\dot Z = \mathrm{T}_{Z}{i}  (\dot X, \dot Y, \dot H)$ where the parameters satisfies  
\begin{equation}
 (\dot X, \dot Y, \dot H) = \mathrm{T}_{Z}{i}^{-1} \left(P_{T_{Z}} F(Z) \right) 
 \label{eq:projchart}
 \end{equation}
with 
$$
P_{T_{Z}} F(Z)  = P_U^\perp F(Z) P_V^T+   P_U^TF(Z) (P_V^\perp)^T+ P_U F(Z) P_V^T.
$$ 
Here $P_U= UU^+, P_V = VV^+$  denote  the projections associated to $U,V$ respectively, and their related orthogonal projections $P_U^\perp = I-P_U$,  $P_V^\perp = I-P_V$.
Equation \eqref{eq:projchart} yields equivalently to the following system of ODEs on the parameters: 
\begin{equation}
\begin{array}{rcl}
\dot X&=& U_\perp^+ F(Z) (V^+)^T G^{-1}, \\ 
\dot Y&=& V_\perp^+  F(Z)^T(U^+)^T G^{-T},\\ 
\dot H&=& U^+ F(Z) (V^+)^T.
\end{array}
\label{eq:dynsysfactorchart}
\end{equation} 

The proposed geometrical description ensures that  \eqref{eq:projchart} admits a unique maximal solution $Z$ when the original problem is an autonomous dynamical system with {\it vector field} $F$ \cite[Theorem 2.6, Theorem 3.5]{Falco2018}
\begin{proposition} \label{prop:wellposed}
{Assume $F : \RR^{n\times m} \to \RR^{n\times m}$ is a ${\cal C}^{p}$ vector field.} Then \eqref{eq:dynsys}  admits a unique solution. Moreover, $\mathrm{T}_{Z}{i}^{-1} P_{T_{Z}} F:  \RR^{n\times m} \to \mathbb{R}^{(n-r)\times r} \times \mathbb{R}^{(m-r)\times r}\times \mathbb{R}^{r\times r}$ is a ${\cal C}^{p}$ vector field which ensures that {the dynamical system \eqref{eq:projchart} also admits a unique solution.}
\end{proposition}

\subsection{Link with the geometric description introduced in  \cite{Koch2007}.} \label{sec:classical}

In this section, we discuss the relation between the proposed geometric approach and the description introduced in  \cite{Koch2007} .
In the latter, the non-uniqueness of the parametrization $Z = UGV^T$ is avoided by computing the tangent matrix $\delta Z$ in $T_Z \Mc_r(\RR^{n \times m})$. Introducing 
\begin{equation}
\delta Z = \delta U GV^T + U \delta G V^T + U  G \delta V^T 
\label{eq:tankoch}
\end{equation}
 together with the {\it gauge conditions}
\begin{equation}
U^T \delta U = 0  \text{ and } V^T \delta V =0,
\label{eq:gaugekoch}
\end{equation}
and assuming that $U,V$ are orthogonal, then the parameter derivatives $\delta U, \delta G, \delta V$ are uniquely given \cite[Proposition 2.1]{Koch2007} by 
\begin{equation}
\begin{array}{rcl}
\delta U&=& (I-UU^T) \delta Z V G^{-1}, \\ 
\delta V&=& (I-VV^T) \delta Z^TU G^{-T},\\ 
\delta G&=& U^T \delta Z V.
\end{array}
\label{eq:factkoch}
\end{equation}
This can be interpreted as constructing an isomorphism between 
$$
T_{Z} \Mc_r(\RR^{n \times m}) ~ \text{ and }
\{(\delta G, \delta U, \delta V) \in \RR^{n\times r} \times \RR^{m\times r}\times \RR^{r\times r} : U^T \delta U = 0 , V^T \delta V =0\}.
$$
The geometric description proposed in  Section  \ref{sec:geom-description1} allows to recover \eqref{eq:tankoch}-\eqref{eq:factkoch}. Indeed, by setting $(\delta U,\delta V,\delta G) = (U_\perp  \delta X,V_\perp  \delta Z,\delta H)$ in the definition \eqref{eq:tanchart} of the tangent matrix $\delta Z$, we get \eqref{eq:tankoch}. Moreover, the gauge conditions are naturally satisfied as $U^T\delta U = U^T U_\perp \delta X =0, V^T\delta V = V^T V_\perp \delta Z =0$. Finally, for $U,V$ orthogonal we have $UU^+ = UU^T, VV^+=VV^T$ and by multiplying the first and second equations of \eqref{eq:factchart} by $U_\perp$ and $V_\perp$ respectively we recover \eqref{eq:factkoch}. \\

Going back to dynamical low-rank approximation, when $Z$  corresponds to the rank-$r$ approximation of a matrix $A$ through Equation \eqref{eq:proj}, we get from \eqref{eq:factkoch} the following system of ODEs governing the evolution of the factors $U,G,V$, 
\begin{equation}
\begin{array}{rcl}
\dot U&=& (I-UU^T) F(Z) V G^{-1}, \\ 
\dot V&=& (I-VV^T)  F(Z)^TU G^{-T},\\ 
\dot G&=& U^T F(Z) V. 
\end{array}
\label{eq:projoch}
\end{equation}
Again, this system can be deduced from \eqref{eq:dynsysfactorchart} by setting $(\dot U,\dot V,\dot G) = (U_\perp  \dot X,V_\perp  \dot Z,\dot H)$ and assuming that $U,V$ orthogonal.

\section{Projection splitting integrator schemes}\label{sec:algos}

In this section, we derive suitable schemes for numerical integration of the projected equation \eqref{eq:proj}. 
Two splitting methods are presented, first in an abstract semi-discretized framework in Section \ref{sec:abstract_algo}  and then in their practical form in Section \ref{sec:pract_algo}. Then, the relation between these two practical algorithms is discussed.

\subsection{Splitting integrators} \label{sec:abstract_algo}


\subsubsection{Symmetric splitting method} \label{ref:KSL_continu}

We first consider the setting of the classical description \cite{Koch2007}  detailed in Section
\ref{sec:classical}. To perform time integration, a symmetric Lie-Trotter splitting method \cite{Lubich2014} is applied to Equation
\eqref{eq:DF}. {This integration scheme relies on a decomposition of the projection 
$P_{T_Z(t)}$ as follows 
\begin{equation}
P_{T_Z(t)} = {Q}_1 - {Q}_2  + {Q}_3
\label{eq:splitting_op_KSL},
\end{equation}
where ${Q}_1$,  ${Q}_2$ and  ${Q}_3$ are three projections respectively defined by
\begin{equation}
{Q}_1A =  AP_V^T, \quad {Q}_2A = P_U A P_V^T, \quad {Q}_3 A =  P_UA,
\end{equation}
 %
 for any $A \in \RR^{n \times m}$. Note that this splitting is not associated to a direct sum decomposition of the tangent space. Using this splitting, one integration step from $t_0$ to $t_1$} starting from
the factorized rank-$r$ matrix $Z_0=Z(t_0)$ under the form
$Z_0 = U_0 G_0V_0^T$ reads as follows. 


{
%
\begin{enumerate}
\item Integrate on $[t_0,t_1]$ the $n \times r$ matrix differential equation
$$\dfrac{d}{dt} (UG) = F( (UG)V^T)V, \quad \dot V = 0,$$
with initial conditions 
 $ (UG)(t_0) = U_0G_0$, $ V(t_0)=V_0.$
Then set $U_1$ and $\hat G_1$ such that $ U_1\hat G_1 = (UG)(t_1)$.
\item
Integrate on
  $[t_0,t_1]$ the $r\times r$ matrix differential equation
$$\dot G = - U^T F(UGV^T)V, \quad \dot U=0 , \quad \dot V = 0, $$
with initial conditions $G(t_0) = \hat G_1$, $U(t_0)= U_1$,  $V(t_0)=V_0.$
Set $\tilde G_1 =G(t_1)$.
\item
Integrate on
  $[t_0,t_1]$ the $m\times r$ matrix differential equation
$$\dfrac{d}{dt} (VG^T) = F(U(VG^T)^T)^T U, \quad \dot U = 0,$$
with initial conditions $(VG)^T(t_0) = V_0\tilde G_1^T$, $U(t_0)= U_1$.
Set $V_1$ and $G_1$ such that $V_1 G_1^T = (VG^T)(t_1)$.
\end{enumerate}
}
{
After these steps, the obtained approximation is $Z(t_1) = U_1G_1V_1^T$.  Each step corresponds to the integration of the right hand side of \eqref{eq:proj} associated to the projections $Q_1$, $Q_2$ and $Q_3$ respectively (for details see \cite[Lemma 3.1]{Lubich2014}).  
}
\begin{remark}
This splitting algorithm works in the following order. First, it updates $UG$, then $G$ and finally $VG$. As we will discuss in  Section \ref{sec:pract_algo}, this particular choice allows to recover exactness properties of the splitting scheme in the context of matrix approximation  \cite{Lubich2014}.
\end{remark}

\begin{remark}
The inversion of $G$ is avoided. This
  convenient choice allow to deal with the case of over-approximation, in
i.e. when the rank of the approximation $Z$ is
  smaller then $r$.
\end{remark}

\subsubsection{Chart based splitting method} \label{sec:chart_splitting}

The second contribution of this paper is to propose a numerical integrator relying on {the fibre bundle structure} of the manifold of fixed rank matrices proposed in Section
\ref{sec:geom-description}. \\

{Based on the chart description of Section \ref{subsec:dlr}, the guiding idea is to perform some update of the parameters $(X,Y,H)$, instead of $(U,G,V)$ directly, whose dynamic is governed by the system of ODEs \eqref{eq:dynsysfactorchart}. Working in a fixed neighborhood ${\cal U}_Z$ of $Z$, the matrices $U,V,G$ are fixed and Equation \eqref{eq:dynsysfactorchart} writes equivalently 
\begin{equation}
\begin{array}{rcl}
\dot H&=& U^+ F(Z) (V^+)^T,\\[0.2cm]
\dot X G &=& {U}_\perp^+ F(Z) (V^+)^T, \\[0.2cm]
\dot Y G^T&=& {V}_\perp^+ F(Z)^T(U^+)^T.
\end{array}
\label{eq:dynsysfactorchartmodif}
\end{equation}
Then, we integrate the system \eqref{eq:dynsysfactorchartmodif} from $t_0$ to $t_1$ in three steps. Letting $U(t) = U(t_0) + U(t_0)_\perp X(t)$, $V(t) =V(t_0) + V(t_0)_\perp Y(t)  $ and $Z(t) = U(t) H(t) V(t)^T$, we start  from $(X(t_0),Y(t_0),H(t_0))=(0,0,G(t_0))$ and we proceed as follows. 
%
%
}
{
\begin{enumerate}
\item
Integrate on $[t_0,t_1]$ the $r \times r$ matrix differential equation
$$
\dot H = U^+F( Z)(V^+)^T,  \; \dot X =0,  \quad \dot Y=0,
$$
{with initial conditions $X(t_0)=0, H(t_0) = G(t_0)$ and $Y(t_0)=0$. } Set $ H_1 = H(t_1)$.  
\item
Integrate on $[t_0,t_1]$
  the $n\times r$ matrix differential equation
$$
\dot X H=  U_\perp^+ F(Z)(V^+)^T,  \quad {\dot H =0}, \quad \dot Y=0,
$$
{with initial conditions $X(t_0)=0, H(t_0) =   H_1$ and $Y(t_0)=0$. } 
Then set $ X_1 = X(t_1)$.
\item  
Integrate on
  $[t_0,t_1]$ the $m\times r$ matrix differential equation
$$
\dot Y H^T  = V_\perp^+ F(Z)^T(U^+ )^T,  \quad  {\dot H=0}, \quad \dot X=0.
$$
with initial conditions $X(t_0)=X_1, H(t_0) = H_1$ and $Y(t_0)=0$. 
Then set $Y_1 = Y(t_1)$.
\end{enumerate}
After these three steps, we obtain an approximation $Z(t_1) := U_1 H_1V_1 ^T$ with $U_1 = U(t_0) + U(t_0)_\perp X_1$ and 
$V_1 = V(t_0) + V(t_0)_\perp Y_1$. \\
}

In the lines of the previous method (see Section \ref{ref:KSL_continu}),
the chart based method {can be interpreted as a Lie-Trotter splitting that relies on the following decomposition
of the projection  
\begin{equation}
P_{T_Z(t)} = { P}_1 + { P}_2  + { P}_3
\label{eq:splitting_op_chart},
\end{equation}
where ${ P}_1$,  ${ P}_2$ and  ${ P}_3$ are three projections respectively defined by
\begin{equation}
{  P}_1A =  P_U A P_V^T, \quad {  P}_2A = P_U^\perp AP_V^T, \quad {  P}_3 A =  P_U A (P_V^\perp)^T,
\end{equation}
%
for any $ A \in \RR^{n \times m}$. Note that, contrary to the symmetric splitting method, the proposed splitting follows from a direct sum decomposition of the tangent space.}
Here each term $P_i$ of the projection is associated to the ODE solved at Step $i$. This point is discussed and detailed in the Appendix \ref{A:1}. \\


\begin{remark} 
As for the splitting method of Section \ref{ref:KSL_continu}, we avoid the inversion of matrix $G$ which allows to deal with overapproximation case where $H$ is singular. 
\end{remark}

{
\begin{remark} \label{rem:sign}
The symmetric splitting and the chart based methods differ by the update order. Indeed, the chart based method first updates $H$ and then $X,Y$ (or equivalently $G$ and then $U,V$). Moreover, Step 2 of the symmetric splitting method described in Section \ref{ref:KSL_continu} can be interpreted as a backward evolution problem  that can be ill-conditionned, as pointed out in \cite[Section 5]{Bachmayr2020}.  In the chart based method, the update of $G$ at Step 1 is still a forward evolution problem due to our splitting choice.  Let us mention that other integration strategies have been recently proposed in \cite{Bachmayr2020,Ceruti2020}. 
\end{remark}
 }
 
\subsection{Practical algorithms}  \label{sec:pract_algo}

Now, we provide practical formulation of those methods amenable for numerical use. To that goal, let introduce preliminary notations. We consider a uniform discretization of the time interval $[0,T]$, $T>0$, containing $K+1$ nodes $0 = t^0<t^1<t^2<\dots<t^K$ where $t^k = k \Delta t$ and $\Delta t = \frac{T}{K}$. The matrix $Z^k$ is the approximation at each time step $t^k$ of $Z$ computed on $[t_0,t_1] = [t^{k},t^{k+1}]$ through the splitting methods given in Section \ref{sec:abstract_algo}. Explicit approximation of the flux is performed leading the schemes summarized in the following algorithms.\\

{Following the previous sections, we first introduce the {\it KSL Algorithm} (see Algorithm \ref{alg:KSL}) for the symmetric splitting scheme. Here, we adopt the original name given in \cite{Lubich2014}, with the following correspondence: $K$ stands for $UG$, $S$ for $G$ and $L$ for $VG^T$. 
\begin{algorithm}[KSL algorithm]\label{alg:KSL}
Given the initial rank-$r$ approximation $Z^0 = U^0G^0(V^0)^T$, compute $Z^k \in \Mc_r(\RR^{n \times m})$ for $k \in \{1,...,K\}$ as follows. 
\begin{itemize}
\item Start with $U_0=U^{k-1}, G_0=G^{k-1}$ and $V_0=V^{k-1}$.
\begin{enumerate}
\item Set 
$$
(U G)_1  =  U_0 G_0  + \Delta t F(Z_0,t^{k-1})V_0 
$$
and $U_1$ and $\hat G_1$ such that $U_1 \hat G_1 = (UG)_1$ (computed using QR).
\item Set
$$
\tilde G_1 = \hat G_1 -\Delta t U_1^T F(U_1 \hat G_1 V_0^T, t^{k-1})V_0.
$$
\item Set 
$$
(VG^T)_1 = V_0 \tilde G_1^T + \Delta t F(U_1 \tilde G_1 V_0^T, t^{k-1})^TU_1^T  
$$
and $V_1$ and $G_1$ such that $V_1 G_1^T = (VG^T)_1$ (computed using QR).
 \end{enumerate}
\item  Set $Z^k =  U_1 G_1 V_1^T$.
\end{itemize}
\end{algorithm}

Now, we give the {\it chart based algorithm} (see Algorithm \ref{alg:1chart}) relying on the splitting introduced in Section \ref{sec:chart_splitting}. 
\begin{algorithm}[Chart based algorithm]\label{alg:1chart}
Given the initial rank-$r$ approximation $Z^0 = U^0G^0(V^0)^T$, compute $Z^k \in \Mc_r(\RR^{n \times m})$ for $k \in \{1,...,K\}$ as follows. 
\begin{itemize}
\item Set $U_0= U^{k-1},V_0=V^{k-1}$ and $H_0=G^{k-1}$.
\begin{enumerate}
\item Set
$$
\hat H_1 =  H_0+ \Delta t  (U_0)^+ F(U_0 G_0 V_0^T, t^{k-1})(V_0^+)^T.
$$
\item Set  
$$
(XH)_1 = \Delta t \, {U_0}_\perp^+ F(U_0 \hat H_1 V_0^T,t^{k-1})(V_0^+)^T
$$
and $U_1$ and $\tilde H_1$ such that 
$U_1 \tilde H_1 = U_0 \hat H_1  + U_{0,\perp} (XH)_1$ (computed using QR).
\item Set
$$
 (YH^T)_1  = \Delta t \,  {V_0}_\perp^+ F(U_1 \tilde H_1V_0^T,t^{k-1})^T(U_1^+)^T, 
$$
and $V_1$ and $H_1$ such that ${V_1 H_1^T  }  = V_0 \tilde H_1^T   +  V_{0,\perp}  (YH^T)_1$ (computed using QR). 
\end{enumerate}
\item Set $Z^k =  U_1 H_1 V_1^T$.
\end{itemize}
\end{algorithm}
}

{
\begin{remark}
For numerical stability, updates of $U$ and $V$ are performed with $QR$ factorization in both algorithms. Note that $SVD$ should have been considered instead. 
\end{remark}
\begin{remark}
From practical point of view, the new algorithm is very similar  to the KSL algorithm in its implementation. The only practical differences are in the order of the update and the sign in the update of $G$ at Step 1, as discussed in Remark \ref{rem:sign}.
\end{remark}
\begin{remark}
The proposed practical algorithms are only first order splitting methods. Higher order extensions should be considered in the line of \cite[Section 3.3.]{Lubich2014}.
\end{remark}
\begin{remark}
In the chart based algorithm, matrices  ${U_0}_\perp$ and ${V_0}_\perp$ are not computed explicitly in practice. 
Indeed, we directly compute $U_{0,\perp}(X H)_1$ and $V_{0,\perp}(Y H^T)_1$ that only requires the calculation of  $P_{U_0}^\perp = U_{0,\perp}U_{0,\perp}^+ = I - P_{U_0}$ and $P_{V_0}^\perp = V_{0,\perp}V_{0,\perp}^+ = I - P_{V_0}$. 
\end{remark}
}

\paragraph{{Link between the two practical algorithms}}
Here, we discuss the similarity of the two proposed practical algorithms in the particular case where the flux $F$ is independent of $Z$
which includes the particular case of matrix approximation where $F(Z(t),t) := \dot A(t)$ with $A(t) \in \RR^{n \times m}$. 

\begin{lemma}\label{prop:similar}
In the case where the flux $F$ is independent of $Z$ meaning 
$
F(Z(t),t) := F(t),
$
 the sequence of approximations $\{Z^k\}_k \subset \Mc_r(\RR^{n\times m})$ provided by KSL algorithm and chart based algorithm coincide. 
\end{lemma}

\begin{proof}

We assume that the explicit flux evaluation $F(t^{k-1})$, required at each time step $k$ of both KSL and chart based algorithms, is noted $F^{k-1}$  for the sake of presentation.  {For the sake of comparison, the two algorithms are reformulated as follows in term of successive updates on $U,G,V$ and $K,L$. Approximations provided by the chart based algorithm are surrounded by a $\bar{}$ symbol.}  \\
{
At step $k$, the KSL algorithm provides  the following approximations 
\begin{eqnarray}
K_1 =  U_0 G_0+\Delta t F^{k-1}V_0,  & (U_1,\hat G_1) = QR(K_1), \label{eq:KSL1}\\ 
\tilde G_1 = \hat G_1 - \Delta t  U_1^T F^{k-1} V_0,& \label{eq:KSL2}\\ 
L_1 =  V_0 \tilde G_1^T+ \Delta t (F^{k-1})^TU_1,  &  (V_1,G_1^T)  = QR(L_1). \label{eq:KSL3}
\end{eqnarray}
Meanwhile, the chart based algorithm gives
\begin{eqnarray}
\bar {\hat G}_1 = G_0 + \Delta t  U_0^T  F^{k-1} V_0,  & \label{eq:c1}\\
\bar K_1 = U_0 \bar {\hat G}_1+\Delta t  (I-P_{U_0}) F^{k-1} V_0,  & (\bar U_1, \bar{\tilde G}_1)=QR(\bar K_1), \label{eq:c2}\\
\bar L_1= V_0\bar{\tilde G}_1^T+ \Delta t  (I-P_{V_0}) (F^{k-1})^T\bar U_1,  & (\bar V_1,\bar G_1^T)=QR(\bar L_1) \label{eq:c3}.
\end{eqnarray}
Expending $\bar {\hat G}_1$ in Equation \eqref{eq:c2} yields 
$$\bar K_1= U_0 \bar {\hat G}_1+\Delta t  (I-P_{U_0}) F^{k-1} V_0 =  U_0G_0+\Delta t  U_0U_0^T  F^{k-1} V_0 + \Delta t  (I-U_0U_0^T) F^{k-1}  V_0  = K_1.$$
Then $\bar{\tilde G}_1 = \hat G_1$ and $\bar U_1 = U_1$. Using these equalities and injecting the expression of $L_1$, we get 
$$  
L_1 = V_0 \tilde G_1^T+ \Delta t (F^{k-1})^TU_1= V_0 \bar{ \tilde G}_1^T- \Delta t \bar V_0 \bar V_0^T(F^{k-1})^T \bar U_1  + \Delta t (F^{k-1})^T \bar U_1 = \bar L_1$$
from which we deduce $L_1 = \bar L_1$ then $G_1 = \bar G_1$ and $V_1 = \bar V_1$. This implies that $Z^k = \bar Z^k$ which concludes the proof. 
}
\end{proof}

In \cite{Lubich2014}, under the assumption that the set of matrices of bounded rank $r$ is an invariant set for the flow, it follows that the KSL method is first order accurate and exact. In consequence, under the same assumptions from Lemma~\ref{prop:similar} the chart based method satisfies the same properties.

\begin{proposition}\label{prop:exactness}
Assume that $A(t) \in \RR^{n \times m}$ is a matrix with bounded rank $r$ for all time $t$ and $F(Z(t),t) := \dot A(t).$ Setting $\Delta t F(t^{k-1}) = A(t^k)-A(t^{k-1})$, 
the chart based splitting algorithm is exact, that is, $Z^k = A(t^k).$ 
\end{proposition}

\begin{proof}
The result follows from Lemma \ref{prop:similar}  and the property of exactness of the KSL algorithm \cite[Theorem 4.1]{Lubich2014} in that context. 
\end{proof}

\begin{remark}
Proposition \ref{prop:exactness} insures that both methods are exact, whereas they rely on two different splitting methods, with different order of parameter updates.  This may sound surprisingly at first, as it has been demonstrated by numerical experiment that \cite[\S 5.2]{Lubich2014} KLS variant \footnote{For that variant, it means that $UG$,  $VG^T$ and then $G$ are updated.}  of the KSL algorithm provides less accurate approximations and loses exactness property.
\end{remark}

\section{Numerical results} \label{sec:num_results}

In this section, we confront the two presented splitting algorithms namely {\bf KSL} and {\bf chart based} algorithms for dynamical low-rank approximation in the context of matrix approximation, in Section \ref{sec:num_matrix}, and matrix dynamical systems arising from a parameter-dependent semi-discretized viscous Burgers's equation, in Section \ref{sec:num_burgers}.

\subsection{Matrix approximation} \label{sec:num_matrix}

As first example, we consider the problem of approximating a matrix $A(t) \in \RR^{100 \times 100}$ on the time interval $[0,1]$. This matrix is given explicitly as 
$$
A(t) =e^{t W_1} De^t e^{tW_2}, 
$$
where $D$ is a diagonal matrix in $\RR^{100 \times 100}$ with non zero coefficients $d_{ii} = 2^{-i}, i\le 10$, and $W_1,W_2 \in $ are two random skew symmetric matrices.
Then, $A(t)$ has a rank $10$ whose non zero singular values are $\sigma_i(t) = e^t 2^{-i}, i \le 10$. \\

The two algorithms are first applied with numerical flux taken equal to 
$
\Delta t F(Z_i,t^{k-1}) = A(t^k)-A(t^{k-1}),  
$
where $Z_i$ stands for the approximation at Step $i$ of the iteration $k$ of the splitting algorithms. 
In what follows, we study the behaviour of the approximation error $e^k= \|A(t^k) - Z^k\|$ at iteration $k$.
{Table \ref{tab:error_evol_DA} gives the values of $\max_k e^k$ for the two methods for $r \in\{10,20\}$. We observe that both methods are exact for $r =10$, which is the rank of the exact solution. In case of over-approximation with $r=20$, the two methods are robust and again provide exact approximation (up to machine precision). 
\begin{table}[H]
\centering 
\begin{tabular}{c|c|c}
& KSL & chart \\
\hline \hline 
$r=10$  & $4.03.10^{-15}$& $5.22 .10^{-15}$ \\
$r=20$ & $5.36.10^{-15}$ & $3.77.10^{-15}$
\end{tabular}
\caption{Matrix approximation with $\Delta t F(Z_i,t^{k-1}) = A(t^{k})-A(t^{k-1})$: maximum over $k$ of $e^k$ for $\Delta t = 5 ~ 10^{-3}$, $r\in \{10,20\}$.
\label{tab:error_evol_DA}}
\end{table}}

Similar experiments are performed but considering  the exact expression of the matrix derivative with $F(Z_i,t^{k-1}) = \dot A(t^{k-1})$. 
Approximation error evolution is depicted on Figure  \ref{fig:error_evol}.

\begin{figure}[H]
\centering 
\includegraphics[height=5cm]{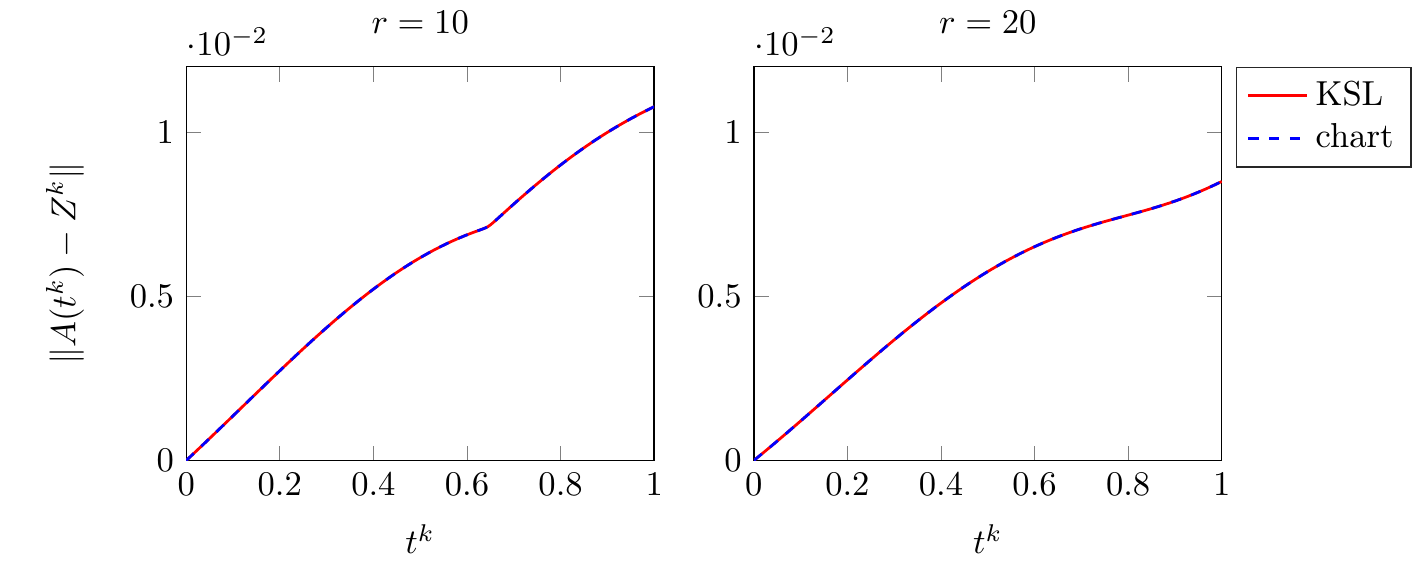}\\
\caption{Matrix approximation with $F(Z_i,t^{k-1}) = \dot A(t^{k-1})$: evolution of $e^k$ for $\Delta t = 5 ~ 10^{-3}$, $r\in \{10,20\}$.
\label{fig:error_evol}}
\end{figure}

Again, both methods give similar results for $r\in \{10,20\}$, as shown on Figure \ref{fig:error_evol} where the error plots coincide either in case of overapproximation. 
However, for this choice of $F(Z_i,t^{k-1})$, both method are no longer exact. Here, the error increases with time up to approximatively $10^{-2}$ with $r=10$ and $10^{-3}$ for $r=20$. 
To quantify this error, we perform some convergence study with respect to $\Delta t$ and $r$. On Figure \ref{fig:errora_rankvsdt_final}, we show the behaviour of the final approximation error 
$e^K=\|A(1)-Z^K\|$ for various rank $r \in \{4,6,8,16,32\}$ and time step $ \Delta t \in \{10^{-5},..., 10^{-1}\}$.

\begin{figure}[H]
\centering 
\includegraphics[height=5cm]{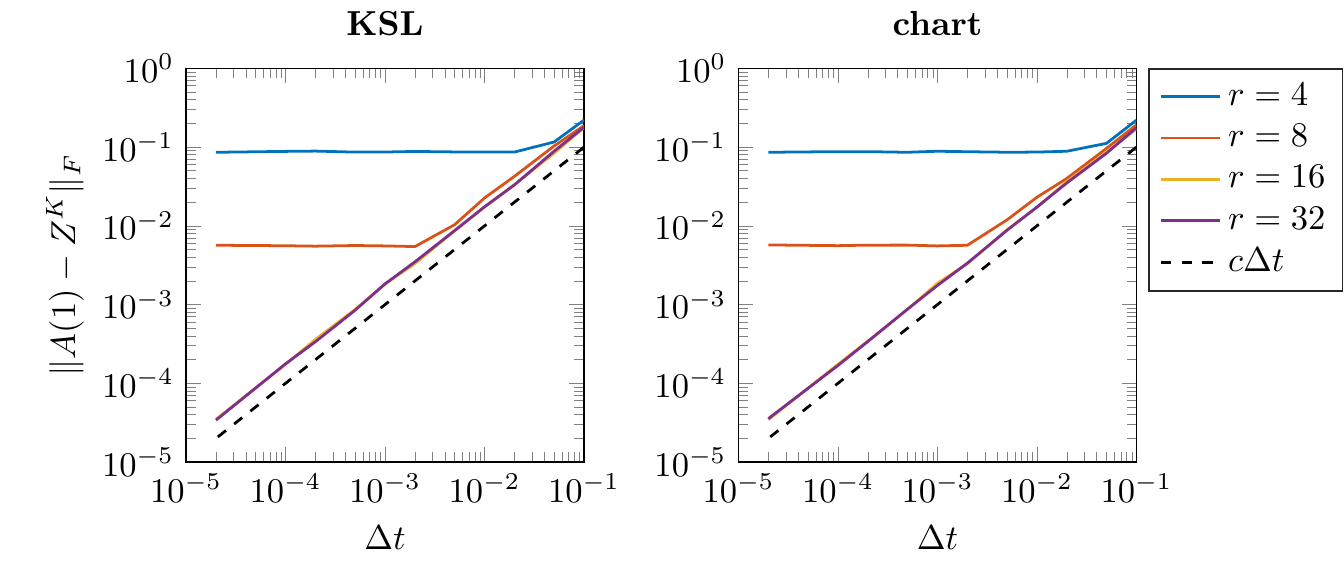}
\caption{Matrix approximation with $F(Z_i,t^{k-1}) = \dot A(t^{k-1})$ : final error $e^K$ for both algorithms for different $(r,\Delta t)$.
\label{fig:errora_rankvsdt_final}}
\end{figure}

Figure \ref{fig:errora_rankvsdt_final} illustrates that the two methods are first order in time as expected. Indeed we observe a linear decreasing of the error up to some stagnation value for $r \le 8$. This stagnation is related to the low-rank approximation error as it decreases when $r$ increases.  

{\subsection{Parameter dependent problem} } \label{sec:num_burgers}

In the lines of \cite{Billaud2017}, we consider the approximation of the parameter-dependent Burger's viscous equation in one dimension. To that goal, let $\Omega\times I = (0,1)\times [0,1]$ be a space time domain. We seek $u(\cdot;\xi)$ the solution of 
\begin{eqnarray}
\partial_t u(x,t;\xi)- \mu(\xi) \partial_{xx}^2 u(x,t;\xi)
+ { u(t,x;\xi) \partial_x u(x,t;\xi)} &=&f(x,t; \xi),\quad \text{ on } \Omega \times I, \label{eq:burgers}
\end{eqnarray}
 with the initial data $u^0(\cdot, \xi): \Omega \to \RR$ and supplemented with  homogeneous Dirichlet boundary conditions. The solution $u(\cdot; \xi)$ depends on the parameter $\xi \in \RR^3 $ through the viscosity $\mu(\xi) = \xi_1 $,  the initial condition and the source term defined by means of the function $f(\cdot,\cdot;\xi) : \Omega \times I \to \RR$ given by
$$
f(x,t;\xi) = \xi_2\exp(-(x-0.2)^2/0.03^2)\sin(\xi_2 \pi t) \boldsymbol{1}_{[0.1,0.3]}(x).                         
$$

The problem \eqref{eq:burgers} is semi-discretized in space by means of finite difference (FD) schemes with $n$ nodes and $m$ instances of the parameter $\xi$ such that we get the following dynamical system 
\begin{equation}
\dot X(t) = L X(t) +  {h(X,t)}, \qquad X(0) = 0, 
\label{eq:appli1}
\end{equation}
where the solution $X(t)$ is a matrix in $ \RR^{n \times m}$. The tensorized operator ${L = D_x\otimes M_\xi}$ is defined by means of 
$D_x \in \RR^{n \times n}$ the  discrete Laplacian obtained by second order centered FD scheme and $M_\xi \in \RR^{m \times m}$ a diagonal matrix whose non-zero coefficients are  the $m$ instances of $\xi_1$.
Moreover, we define the matrix valued function $h$  with entries $[h(Z,t)]_{ij} = Z_{ij}[C_x Z]_{ij} + f(x_j,t; \xi_i)$ where $C_x \in \RR^{n \times n}$ is the discrete version of the first derivative obtained by 1st order centered FD scheme.\\

We first confront the KSL and chart based algorithms for solving the discrete parameter dependent Burger's equation when one single parameter varies. Then, the multiple varying parameter problem is studied.\\

\subsubsection{Single parameter Burger's problem}
   
In this section, the parameter $\xi_1$ takes its values in $[0.01, 0.06]$ while the others are fixed to $\xi_2=4$ and $\xi_3 =0$. 
We chose the following initial condition 
$$
u^0(x;\xi)=  sin(x\xi_1) e^{-100(x-10\xi_1)^2}. 
$$
and set $n =100$ and $m=60$. For solving the matrix ODE given by Equation \eqref{eq:appli1}, we first confront the chart based and KSL algorithms for various ranks, fixing $\Delta t = 10^{-4}$. The approximations obtained are compared to a reference solution noted $\{X_{ref}^k\}_{k=0}^K$ computed with an  explicit Euler scheme with $\Delta t = 5.10^{-6}$. The time step $\Delta t$ is setting small enough to ensure numerical stability  since an explicit integration scheme is applied.\\

Figure \ref{fig:profile_chart} illustrates the behaviour of the numerical solution obtained with the chart based algorithm at final time $t=1$, and $r=20$ for two different instances of the parameter $\xi$.
\begin{figure}[H]
\centering 
\begin{tabular}{ccc}
\includegraphics[height=3.5cm]{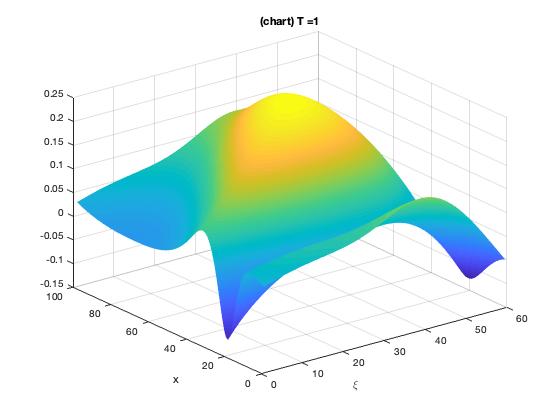} &
\includegraphics[height=3.5cm]{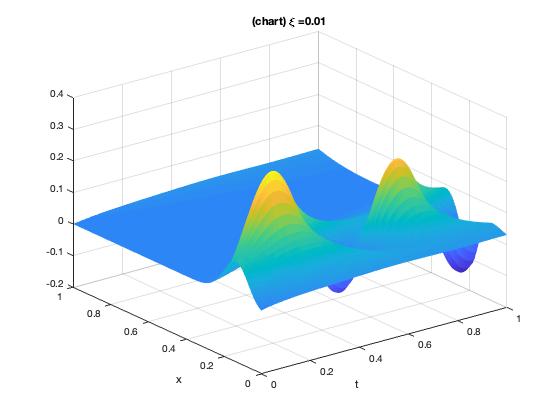}&
\includegraphics[height=3.5cm]{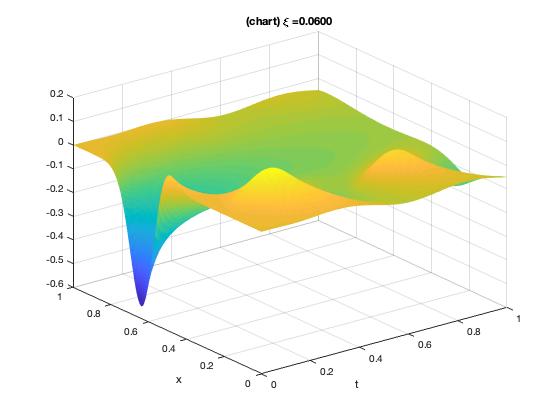}\\
$t=1$ & $\xi_1= 0.01$ & $\xi_1= 0.06$ 
\end{tabular}
\caption{Burgers's equation: approximation for the chart method with $r=20$ and $\Delta t = 10^{-4}$, at final time $t=1$ (left), for $\xi_1=0.01$ (center) and $\xi_1=0.06$ (right).  \label{fig:profile_chart}}
\end{figure}

As we can observe on Figure \eqref{fig:profile_final}, the approximations computed with the two methods are in good agreement with the reference solution.
\begin{figure}[H]
\centering 
\includegraphics[height=5cm]{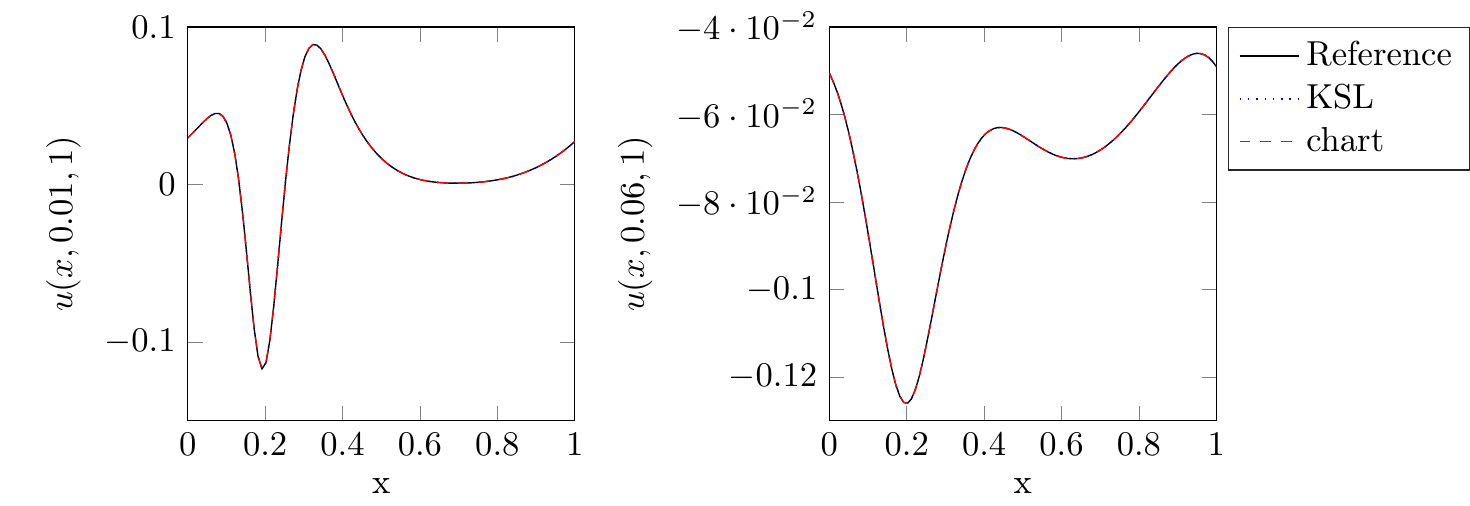}
\caption{Burgers's equation: approximation for the chart based and KSL method compared to the reference solution at final time for $r=20$ and $\Delta t = 10^{-4}$ . \label{fig:profile_final}}
\end{figure}

{
Now, we compare the approximation error to the reference noted $e^k = \|X_{ref}^k-Z^k\|$ for the two splitting algorithms. On Figure \ref{fig:chart_vs_KSL_nl_errora_modif_100_toref}, the evolution of $e^k$ for both algorithms is studied for different ranks. As we can observe, both methods seem to provide an approximation with similar accuracy, except for ranks $15$ and $20$ where the chart based algorithm provides a slightly more accurate approximation than the KSL algorithm. Note that the error does not necessary increase with respect to time as classically observed for DLR approximation methods {which is related to the particular problem considered here}. \\
}

\begin{figure}[H]
\centering 
\includegraphics[scale=1]{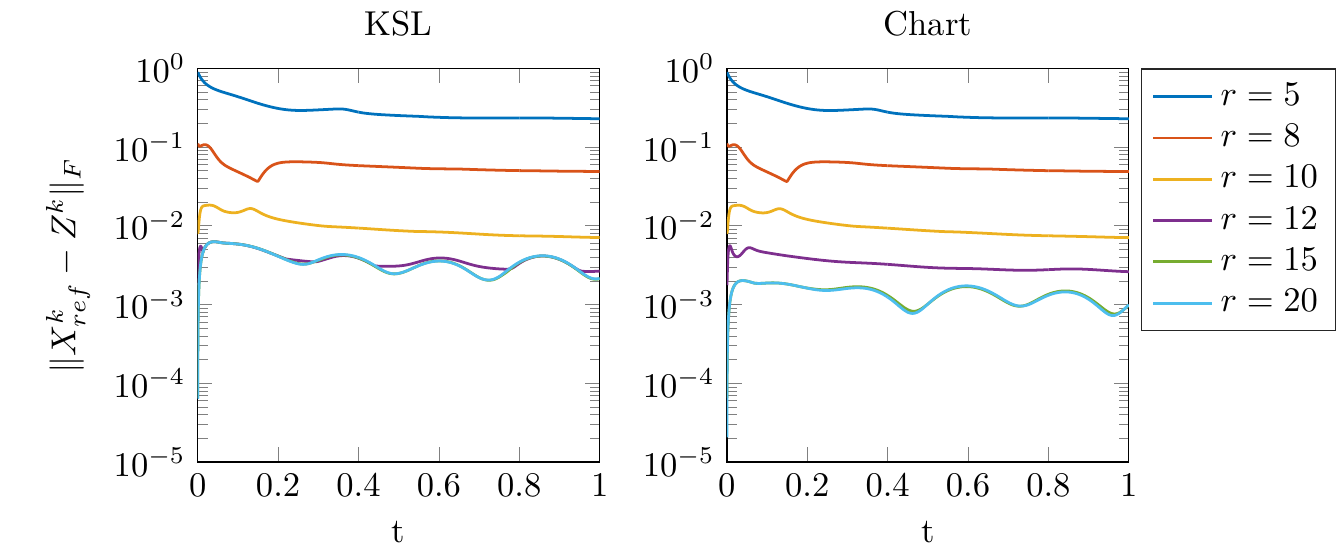}
\caption{Burgers's equation: evolution of $e^k$ for both algorithms for  {$r \in \{5,8,10,12,15,20\}$}. \label{fig:chart_vs_KSL_nl_errora_modif_100_toref}}
\end{figure}

{This observation is confirmed by Figure \ref{fig:errora_rankvsdt_final_100_toref} where  we analyze the convergence with respect to rank and time step. As we  can see, the error $e^K = \|X_{ref}^K-Z^K\|$ decreases with respect to the rank and time step, and  is the smallest (by up to one order of magnitude) for the chart based algorithm for larger rank. }

\begin{figure}[H]
\centering \includegraphics[scale=1]{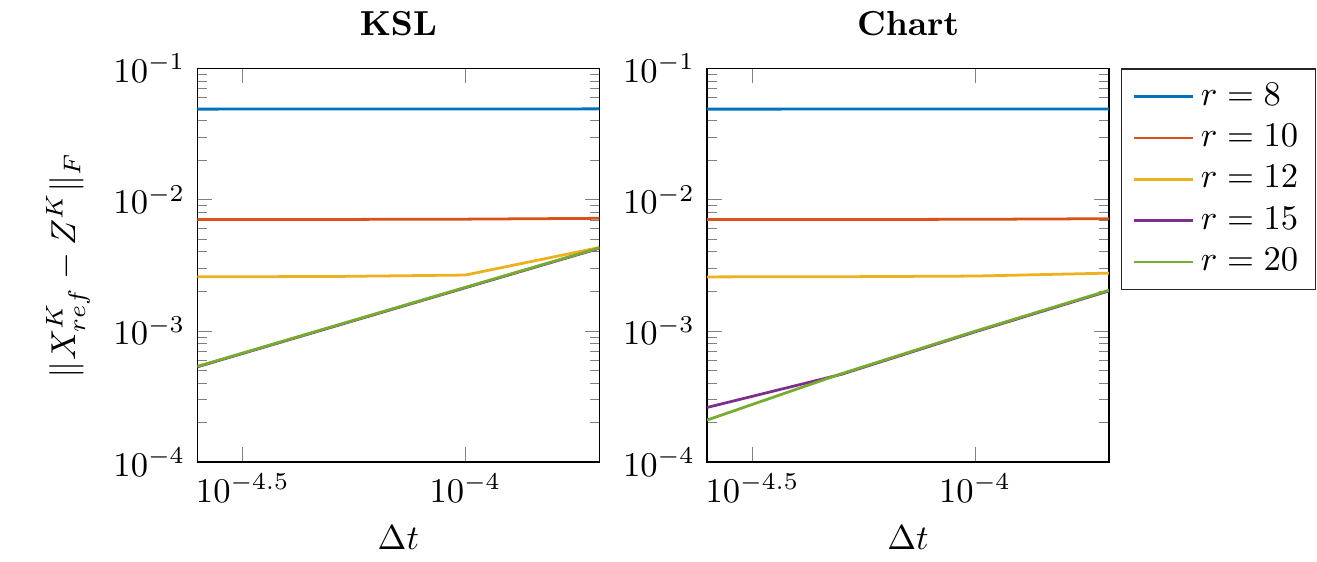}
\caption{Burgers's equation:  error $e^K$ at final time for different $(r,\Delta t)$. \label{fig:errora_rankvsdt_final_100_toref}}
\end{figure}

\subsubsection{Multiple varying parameters Burger's problem}

To conclude this section, we illustrate the behavior of the two methods for the case where $\xi$ is a vector of independent random parameters uniformly distributed on $[0.01,0.06]\times [2,4] \times [0.01,0.1]$.Here, the initial condition is
$$
u^0(x;\xi)=  \xi_3 e^{-100(x-10\xi_2)^2}. 
$$ 
The numerical simulations  are performed for $m=60$ and $n=100$. Here, the two KSL and Chart methods are run with $\Delta t = 10^{-4}$ and compared to true numerical reference solution $\{X_{ref}^k\}_{k=0}^K$  obtained with the explicit Euler scheme with the same time step.
The approximation of numerical solution for the Burger's problem computed with the Chart method with $r=20$ is plotted for two distinct realizations of the parameter $\xi$ showing different features with respect to parameter.
\begin{figure}[H]
\centering 
\begin{tabular}{ccc}
\includegraphics[height=4cm]{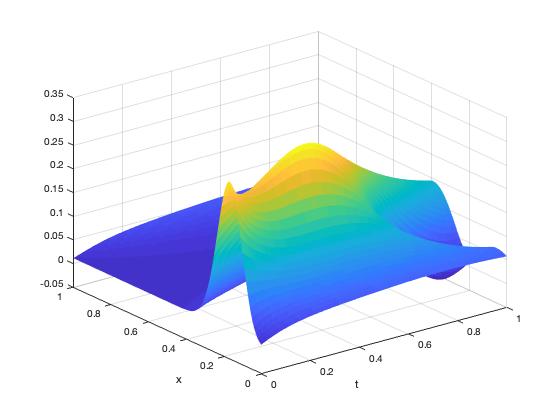}&
\includegraphics[height=4cm]{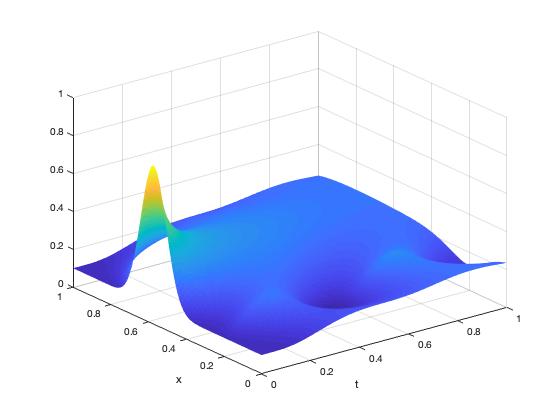}
\end{tabular}
\caption{Burgers's equation: approximation for the chart method with $r=20$ and $\Delta t = 10^{-4}$ for two instances of $\xi$.  \label{fig:profile_nl2}}
\end{figure}
We represent on Figure \ref{fig:multiple_final_errora_vs_rank} the approximation error at final time for the three approaches with respect to $r\in \{5,10,12,15, 20,25,30\}$. We clearly observe that the chart  method provides a better approximation than KSL  for which the error seems to stagnate after $r\ge 12$.
\begin{figure}[H]
\centering \includegraphics[scale=1]{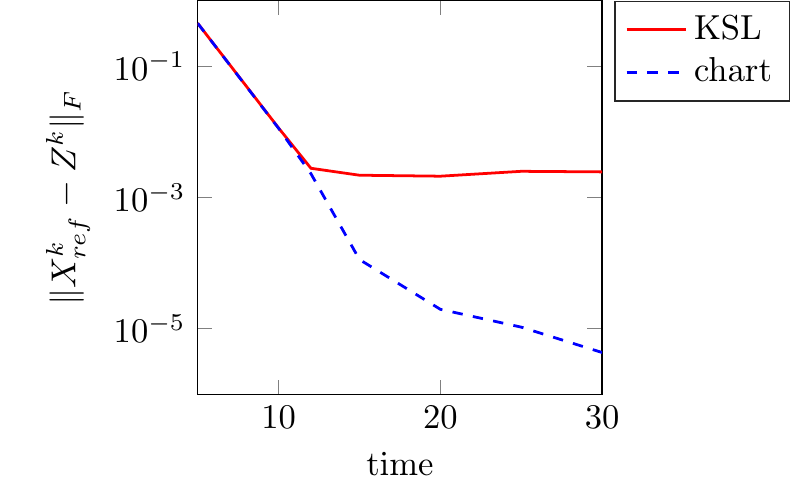}
\caption{Burgers's equation (multiple parameter case): error $e^K$ at final time for different $r$. \label{fig:multiple_final_errora_vs_rank}}
\end{figure}

\section{Conclusion}

In this paper, we have introduced and compared some geometry based algorithms for dynamical low-rank approximation. Using a different geometry description of the set of fixed rank matrices relying on charts, we generalized the description of \cite{Koch2007}. {Then, from this description we derive a new splitting algorithm motivated by fibre bundle structure of the manifold of fixed rank matrices. The resulting algorithm is proved to coincide with the KSL algorithm \cite{Lubich2014} in the particular case of matrix approximation.} Nevertheless, for more general problems arising from the semi-discretization of parameter-dependent non-linear PDEs, the chart based algorithm seems to outperform the KSL algorithm in some situation. 
Further work should be conducted for derivation of rigorous error bounds in more general cases as well as high order extension using e.g. Strang splitting variant. {Moreover, the proposed splitting scheme is a first step towards designing new algorithms integrating the geometric structure of the manifold of fixed rank matrices, by working in neighborhoods. Especially, deriving a numerical scheme working alternatively in the horizontal and vertical  tangent spaces (as pointed out in Remark \ref{rem:HV}) is the object of future research.
}

\appendix
\section{Chart based splitting integrator} \label{A:1}

{ Following the same lines as in  \cite{Lubich2014}, we justify how the chart based method introduced in Section \ref{sec:chart_splitting} can be interpreted as a splitting scheme relying on the projection decomposition
\eqref{eq:splitting_op_chart} as the sum of three contributions $P_{T_Z(t)}= P_1+ P_2 + P_3 $. One integration  step of the splitting method starting from $t_0$ to $t_1$  with initial guess $Z(t_0) = U(t_0)G(t_0)V(t_0)^T$ proceeds as follows. 
{\begin{enumerate}
\item[(S1)] Find $Z \in {\cal U}_{Z(t_0)}$ on $[t_0,t_1]$ such that $\dot Z =  P_{U} F(Z) P_{V}^T$ with initial condition $Z(t_0)$.
\item[(S2)] Find $Z \in {\cal U}_{Z(t_0)}$  on $[t_0,t_1]$ such that  $\dot Z =  P_{U}^\perp F(Z) P_{V}^T$ with initial condition given by final condition of step (S1).
\item[(S3)] Find $Z \in {\cal U}_{Z(t_0)}$  on $[t_0,t_1]$ such that  $\dot Z  =  P_{U} F(Z) (P^\perp_{V})^T $ with initial condition given by final condition of step (S2).
\end{enumerate}
}
}

{
At each step (S$i$) of the splitting, $Z$ belongs to the neighborhood  of $Z(t_0)$. Thus  it is given by $Z(t) = U(t)H(t)V(t)^T$ with $U(t) = U(t_0)+U(t_0)_\perp X(t)$, $Y(t)= V(t_0)+V(t_0)_\perp Y(t)$ provided by the ODE solved at Step $i$ of the chart based splitting, as stated in the following proposition.
\begin{proposition}
The solution of (S1) is given by $Z $ with 
\begin{equation}
\dot H= U^+ F (Z)V^T, \quad  \dot X =0,\quad  \dot Y = 0.
\label{eq:H}
\end{equation}
with $H(t_0)=G(t_0)$, $X(t_0)=0$ and $Y(t_0)=0$.  Set 
\item Letting $H_1$ be the final condition of $H$ from (S1), the solution of (S2) is given by $Z$ with
\begin{equation}
\dot X H =  {U}_\perp^+ F (Z)(V^+)^T, \quad \dot Y= 0,\quad  \dot H=0,
\label{eq:X}
\end{equation}
with $H(t_0)=H_1$, $X(t_0)=0$ and $Y(t_0)=0$. Set $X_1 = X(t_1)$.  
\item  Letting $X_1$ be the final condition of $X$ from (S2), the solution of (S3) is given by $Z$ with
\begin{equation}
\dot Y H^T = {V}_\perp^+ F (Z) (U^+)^T,  \quad \dot X= 0,\quad \dot H=0,
\end{equation}
with $H(t_0)=H_1$, $X(t_0)=X_1$ and $Y(t_0)=0$.
\end{proposition}
\begin{proof}
For each step (S$i$), $Z$ admits the decomposition 
$$
Z = (U(t_0)+U(t_0)_\perp X) H(V(t_0)+V(t_0)_\perp Y)^T
$$
 with derivative
$$
\dot Z =U(t_0)_{\perp} \dot X H(V(t_0)+V(t_0)_{\perp}Y)^T + (U(t_0)+U(t_0)_{\perp}X) \dot H(V(t_0)+ V(t_0)_\perp Y)^T
$$
$$
+ (U(t_0)+U(t_0)_{\perp}X) H(V(t_0)_{\perp}\dot Y)^T. 
$$
For (S1), the derivative satisfies $\dot Z = P_{U} F(Z) P_{V}^T$. Then, multiplying on the left by $U(t_0)^+$ and on the right by $(V(t_0)^+)^T$ the matrix $\dot Z$ in both expressions leads to
$\dot H= U^+ F(Z) (V^+)^T$ and $\dot X =0, \dot Y = 0.$\\
Now let us turn to (S2). The 
 derivative
satisfies $\dot Z = P_{U}^\perp F(Z) P_{V}^T$. By multiplying on the right by $(V(t_0)^+)^T$, the equality is satisfied if 
$\dot X H =  U_\perp^+ F(Z) (V^+)^T$ and $ \dot Y= 0, \dot H =0.$ The third point of the lemma is obtained from (S3) in the same manner, by multiplying the equation $\dot Z = P_{U} F(Z) (P_{V}^\perp)^T$ on the left by $(U(t_0)+U(t_0)_\perp X_1)^+$ and setting $\dot X =0$, $\dot H=0$. 
\end{proof}
}



\end{document}